\documentclass[bibliography=totoc,parskip,fontsize=10pt]{scrartcl}

\usepackage{hyperref}
\usepackage[margin=1.1in]{geometry}  
\usepackage{graphicx}              
\usepackage{amsmath}               
\usepackage{amsfonts}              
\usepackage{amsthm}                
\usepackage{amssymb}
\usepackage{paralist}
\usepackage[usenames,dvipsnames]{color} 
\usepackage[british]{babel}
\usepackage[numbers]{natbib}
\usepackage[affil-it]{authblk}
\usepackage{xypic}
\usepackage{setspace}                

\usepackage[automark]{scrpage2}		
\clearscrheadfoot
\ohead{}
\lohead{\rightmark}
\rehead{\leftmark}

\ohead{\textbf{\pagemark}}

\setheadsepline{1pt} 

\addtokomafont{section}{\center}

\newtheorem{Thm}{Theorem}
\newtheorem{thm}{Theorem}[section]

\newtheorem{prop}[thm]{Proposition}

\newtheorem*{defi}{Definition}

\newtheorem*{exa}{Example}

\DeclareMathAlphabet\mathbb{U}{fplmbb}{m}{n}

\newcommand{\HH}{\mathbb{H}}
\newcommand{\CC}{\mathbb{C}}
\newcommand{\RR}{\mathbb{R}}     

\newcommand{\ZZ}{\mathbb{Z}}     
\newcommand{\NN}{\mathbb{N}}

\newcommand{\defeq}{\mathrel{\mathop{\raisebox{1.05pt}{\scriptsize$:$}}}=}

\newcommand\opna{\operatorname}
\newcommand\mf{\mathfrak}

\newcommand\mass{\operatorname{mass}}

\addtokomafont{section}{\normalsize \rmfamily\scshape}
\addtokomafont{subsection}{\small \rmfamily\scshape}

\setkomafont{sectioning}{\normalcolor\rmfamily\scshape}

\begin{document}

\title{\Large The growth of the first non-Euclidean filling function of the quaternionic Heisenberg group}
\date{}
\author{\large Moritz Gruber}
\maketitle

\begin{abstract}\noindent\small
\textbf{Abstract. }The filling volume functions of the n-th quaternionic Heisenberg group grow, up to dimension n, as fast as the ones of the Euclidean space. We identify the growth rate of the filling volume function in dimension n+1, which is strictly faster than the growth rate of the $(n+1)$-dimensional filling volume function of the Euclidean space.
\end{abstract}

\begin{abstract}\noindent\small
\end{abstract}

%
%
\pagenumbering{arabic}

\section{Introduction}

In general, isoperimetric inequalities delimit the maximal volume needed to fill a boundary of a given volume. Different types of isoperimetric inequalities arise by specifying restrictions to the boundaries and fillings. An interesting class of such isoperimetric inequalities is formed by the \emph{filling volume functions}. These describe the difficulty to fill Lipschitz cycles by Lipschitz chains.
In \cite{Gruber1} we proved an Euclidean behaviour for the filling volume function in dimension $m+1$ of the $n^{th}$ quaternionic Heisenberg group $H^n_\HH$ for $m\in \{1,...,n-1\}$, i.e. it grows like $l^{\frac{m+1}{m}}$. The known results for the filling volume functions of the complex Heisenberg group $H^n_\CC$ (see \cite{Young1},\cite{YoungII}) suggest a growth like $l^\frac{n+2}{n}$ for the filling volume function of $H^n_\HH$ in dimension $n+1$. Our technique in  \cite{Gruber1} could only confirm the super-Euclidean behaviour without telling the exact growth. Now we are able to state it:

\begin{Thm}\label{Thm1}
Let $H^n_\HH$ be the $n^{th}$ quaternionic Heisenberg group, equipped with a left-invariant Riemannian metric.
Then:  
$$F^{n+1}_{H^n_\HH}(l) \sim l^\frac{n+2}{n}\ .$$
\end{Thm} 

This theorem was part of the author's dissertation \cite{Doktorarbeit} at the Karlsruhe Institute of Technology.


\section{Filling volume functions \& Heisenberg groups}\label{S1}

\subsection{Filling volume functions}\label{SectionFF}

\emph{Filling volume functions} describe the difficulty to fill a \emph{Lipschitz cycle} of a given mass by a \emph{Lipschitz chain}.
In the following let $X$ be a metric space and $m \in \NN$. Further we denote by $\opna{vol}_m$ the $m$-dimensional Hausdorff-measure of $X$
and by $\Delta^m$ we denote the $m$-simplex equipped with an Euclidean metric.

A \emph{Lipschitz $m$-chain} $a$ in $X$ is a (finite) formal sum $a=\sum_j z_j\alpha_j$ of Lipschitz maps $\alpha_j:\Delta^m \to X$ with coefficients $z_j\in \ZZ$. 
The \emph{boundary} of a Lipschitz $m$-chain $a=\sum_j z_j\alpha_j$ is defined as the Lipschitz $(m-1)$-chain 
$$\partial a=\sum_j \big(z_j \sum_{i=0}^m (-1)^i \alpha_{j|\Delta^m_i} \big)$$
where $\Delta^m_i$ denotes the $i^{th}$ face of $\Delta^m$.
A Lipschitz $m$-chain $a$ with zero-boundary, i.e. $\partial a=0$, is called a \emph{Lipschitz $m$-cycle}. 
A \emph{filling} of a Lipschitz $m$-cycle $a$ is a Lipschitz $(m+1)$-chain $b$ with boundary $\partial b=a$. 
We define the $\emph{mass}$ of a Lipschitz $m$-chain $a$ as the total volume of its summands: 
$$\mass (a)\defeq \sum_j z_j \opna{vol}_m(\alpha_j)\ .$$
In the case that $X$ is a Riemannian manifold, the volume of such a summand is given by $\opna{vol}_m(\alpha_j)=\int_{\Delta^m} J_{\alpha_j} \opna d\! \lambda$, where $ \opna d\! \lambda$ denotes the $m$-dimensional Lebesgue-measure and $J_{\alpha_j}$ is the jacobian of $\alpha_j$. This is well defined, as Lipschitz maps are, by Rademacher's Theorem, almost everywhere differentiable.  

Given an $m$-cycle, one is interested in the filling with the smallest mass. In the next step one can vary the cycle and examine how large the ratio between the mass of the optimal filling and the mass of the cycle can get. This leads to a family of invariants of the space $X$, the \emph{filling volume functions}:

\begin{defi}
Let $n \in \mathbb N$ and let $X$ be an $n$-connected metric space.
For $m\le n$ the \emph{$(m+1)$-dimensional filling volume function} of $X$ is given by
$$F^{m+1}_X(l)=\sup_a \inf_{b} \mass(b) \qquad \forall l \in \RR^+,$$
where the infimum is taken over all $(m+1)$-chains $b$ with $\partial b=a$ and the supremum is taken over all  $m$-cycles $a$ with $\mass(a)\le l$.
\end{defi}

As we are mostly interested in the large scale geometry of the space $X$, the exact description of the filling volume functions is of less importance to us. Indeed we only look at the asymptotic behaviour of the functions. We do this by the following equivalence relation, which makes the growth rate of the filling volume functions a quasi-isometry invariant.

\begin{defi}
Let $f,g:\mathbb R^+ \to \mathbb R^+$ be functions. Then we write $f\preccurlyeq g$ if there is a constant $C>0$ with 
$$f(l) \le Cg(Cl)+Cl+C \quad \forall l \in \mathbb R^+.$$
If $f\preccurlyeq g$ and $g\preccurlyeq f$ we write $f\sim g$. This defines an equivalence relation.
\end{defi}
We read the notation $f \preccurlyeq g$ as ``\emph{$f$ is bounded from above by $g$}'' respectively ``\emph{$g$ is bounded from below by $f$}'' according whether we are more interested in $f$ or $g$.

\begin{prop}[{see for example \cite[Lemma 1]{YoungH}}]
Let $X$ and $Y$ be $n$-connected Riemannian manifolds. Then:
$$X \text{ quasi-isometric to }Y \ \Rightarrow\  F^{m+1}_X \sim F^{m+1}_Y \quad \forall m \le n.$$ 
\end{prop}
Any two left-invariant Riemannian metrics on a Lie group are Lipschitz-equivalent. So, by the above proposition, the behaviour of the filling volume functions of a Lie group equipped with a left-invariant Riemannian metric does not depend on the choice of this metric.

Let's look at the example of the filling volume functions of the $n$-dimensional Euclidean space. They were first computed by Federer and Fleming in \cite{FF60}.

\begin{exa}
The filling volume functions of the Euclidean space $\mathbb{E}^n$ are
$$F^{m+1}_{\mathbb{E}^n}(l) \sim l^\frac{m+1}{m} \quad \text{for } \ m \le n-1.$$ 
\end{exa}
This enables us to use the terms \emph{Euclidean, sub-Euclidean} and \emph{super-Euclidean filling volume function} for filling volume functions with the same, strictly slower respectively strictly faster growth rate
than $l^\frac{m+1}{m}$.

The following theorem generalises the Euclidean case to spaces with non-positive curvature. For a proof see \cite{Wenger08}.

\begin{thm}\label{nonpos}
The filling volume functions of an $n$-dimensional Hadamard space $X$ are 
$$F^{m+1}_X(l) \preccurlyeq l^\frac{m+1}{m} \quad \text{for } \ m\le n-1.$$ 
\end{thm}

We are interested in this, as the fact that a Riemannian manifold with non-positive curvature has Euclidean or sub-Euclidean filling volume functions in all dimensions yields a sufficient criterion for the existence positive curvature: Let $M$ be a Riemannian manifold with a super-Euclidean filling volume function in some dimension, then it can't be of non-positive curvature.

%
%

\subsection{Heisenberg groups}\label{HG} 

The \emph{complex Heisenberg group} $H^n_\CC$ is a higher dimensional analogue of the classical $3$-dimensional Heisenberg group
$$\boldsymbol{H\hspace{-0.5mm}eis}=\left\{ \begin{pmatrix} 1&x&z \\ 0&1&y \\  0&0&1 \end{pmatrix} \mid x,y,z \in \RR \right\} \le \opna{GL}_3(\RR).$$
Equipped with a left invariant Riemannian metric it appears as a horosphere in the complex-hyperbolic space $\opna{SU}(n,1)/\opna S(\opna U(n) \times \opna U(1))$. As an abstract Lie group it can be defined as follows:

\begin{defi}
The \emph{$n^{th}$ complex Heisenberg group} $H^n_\CC$ is 
the unique simply connected Lie group with Lie algebra $\mathfrak h^n_\CC$ generated by 
$$B_\CC\defeq \{h_1,...,h_n,k_1,...,k_n,K \}$$
and with
$$[k_m,h_m]=K $$
for all $m\in\{1,...,n\}$ and with all other brackets of generators equal to zero.
This is a $2$-step nilpotent Lie group of dimension $2n+1$ with Lie algebra $\mathfrak h^n_\CC=V_1\oplus V_2$, where $V_1=\opna{span}\{h_1,...,h_n,k_1,...,k_n\}\cong\CC^n$ and $V_2=\opna{span}\{K\}\cong\opna{Im}\CC\cong \RR$.
\end{defi}

The group we examine in this paper is the \emph{quaternionic Heisenberg group} $H^n_\HH$. It appears similar to the complex Heisenberg group as a horosphere, now in the quaternionic-hyperbolic space $\opna{Sp}(n,1)/(\opna{Sp}(n)\times \opna{Sp}(1))$.

\begin{defi}
The  \emph{$n^{th}$ quaternionic Heisenberg group} $H^n_\HH$ is 
the unique simply connected Lie group with Lie algebra $\mathfrak h^n_\HH$ generated by 
$$B_\HH\defeq \{h_1,...,h_n,i_1,...,i_n,j_1,...,j_n,k_1,...,k_n,I,J,K \}$$ 
and with
$$   [k_m,j_m]=I=[i_m,h_m], \ [i_m,k_m]=J=[j_m,h_m],\ [j_m,i_m]=K=[k_m,h_m]$$
for all $m\in\{1,...,n\}$ and with all other brackets of generators equal to zero.
This is a $2$-step nilpotent Lie group of dimension $4n+3$  with Lie algebra $\mathfrak h^n_\HH=V_1\oplus V_2$ where $V_1=\opna{span}\{h_1,...,h_n,i_1,...,i_n,j_1,...,j_n,k_1,...,k_n\}\cong\HH^n$ and $V_2=\opna{span}\{I,J,K\}\cong\opna{Im}\HH$, where $\HH$ denotes the Hamilton quaternions.
\end{defi}

These nilpotent Lie groups are representatives of a large class of nicely behaving Lie groups:

\begin{defi}
A simply connected, $d$-step nilpotent Lie group $G$ with Lie algebra $\mf g$ is \emph{stratified}, if there is a decomposition 
$\mf g= V_1 \oplus ... \oplus V_d$
of the Lie algebra with $[V_1,V_j]=V_{j+1}$.
\end{defi}

Let $G$ be a stratified nilpotent Lie group. Then for every $t>0$ there is an automorphism $\widehat s_t : \mf g \to \mf g$ which maps $v\in V_j$ to $t^jv_j$.
We denote the corresponding automorphism $\exp \circ \widehat s_t$ of the Lie group  by $s_t$.

\subsection{A technical criterion for lower bounds}

For our proof of the lower bound on the $(n+1)$-dimensional filling volume function of the quaternionic Heisenberg group $H^n_\HH$, we will apply a technique based on Stokes’ theorem.

\begin{prop}[{see \cite[Proposition 1.2] {Burillo}}]\label{Burillo} 
Let $G$ be a stratified nilpotent Lie group equipped with a left-invariant Riemannian metric and let $m \in \NN$. If there exists a Lipschitz $(m+1)$-chain $b$ and a closed $G$-invariant $(m+1)$-form $\eta$ in G and constants $C,r,s >0$ such that
\begin{enumerate}[]
\item  1) $\mass(s_t(\partial b))\le Ct^r$\ ,
\item  2) $\int_b \eta >0$\ ,
\item  3) $s_t^*\eta = t^s\eta$\ ,
\end{enumerate}
then holds \ $F^{m+1}_G(l) \succcurlyeq l^\frac{s}{r}$.
\end{prop}
Burillo used this criterion to compute lower bounds on the filling volume functions of the complex Heisenberg group $H^n_\CC$. We will use this criterion as well as the property that $H^n_\HH$ contains $H^n_\CC$.

The change-over from the complex to the quaternionic Heisenberg group complicates the construction of the form $\eta$, as there are more non-trivial brackets and therefore $G$-invariant $(n+1)$-forms with $s_t^*\eta = t^{n+2}\eta$ get very rare. In \cite{Doktorarbeit} we have even seen, that in the case of the octonionic Heisenberg group there are no such forms.

%
%

\section{The proof of Theorem \ref{Thm1}}\label{SectionPThm1}

In this section we compute the $(n+1)$-dimensional filling volume function of the quaternionic Heisenberg group $H^n_\HH$. As $H^n_\HH$ fulfils the conditions of \cite[Theorem 1]{Gruber1}, we already have the super-Euclidean upper bound: 
$$F^{n+1}_{H^n_\HH}(l) \preccurlyeq l^\frac{n+2}{n}\ .$$
Thus it remains to prove the corresponding lower bound.

\begin{prop}
Let $H^n_\HH$ be the $n^{th}$ quaternionic Heisenberg group. Then:
$$F^{n+1}_{H^n_\HH}(l) \succcurlyeq l^\frac{n+2}{n}\ .$$
\end{prop} 
\begin{proof}
We are going to use  Proposition \ref{Burillo}. To this end, we have to construct a Lipschitz $(n+1)$-chain $b$ in $H^n_\HH$ and a closed $H^n_\HH$-invariant $(n+1)$-form $\eta$ on $H^n_\HH$ with the correct scaling behaviour.
We start with the constructions Burillo did in the proof of \cite[Theorem 2.1]{Burillo} to obtain the lower bound for the filling volume function $F^{n+1}_{H^n_\CC}$ of the complex Heisenberg group. We denote the there constructed $(n+1)$-chain by $b'$ and the corresponding $(n+1)$-form by $\gamma$.

Now let $B_\CC$ be the usual basis of the Lie algebra $\mf h^n_\CC$ of the complex Heisenberg group $H^n_\CC$ and $B_\HH$ the basis of the Lie algebra $\mf h^n_\HH$ of the quaternionic Heisenberg group $H^n_\HH$ (compare Section \ref{HG}). 
Then the complex Heisenberg group $H^n_\CC$ embeds as a Lie subgroup into the quaternionic Heisenberg group $H^n_\HH$. We do this on the Lie algebra level via the map
$$\Phi:\mf h^n_\CC \to \mf h^n_\HH \quad \text{ defined by } \ h_i \mapsto h_i \ , \ k_i \mapsto k_i \ ,\ K \mapsto K\ .$$
Hence we can consider the  $(n+1)$-chain $b'$ constructed in \cite{Burillo} as an $(n+1)$-chain in the quaternionic Heisenberg group $H^n_\HH$. We set $b\defeq b'$. The above embedding respects the grading of the Lie algebras, i.e. vectors of the first layer are mapped to vectors of the first layer and vectors of the second layer are mapped to vectors of the second layer. Therefore the boundary $\partial b$ of the chain $b$ has the same scaling behaviour in the quaternionic Heisenberg group $H^n_\HH$ as it had in the complex Heisenberg group $H^n_\CC$, i.e. $\mass(s_t(\partial b)) \le \mass(\partial b) \cdot t^n$.

Consequently, it remains to construct a closed, $H^n_\HH$-invariant $(n+1)$-form $\eta$ on $H^n_\HH$, such that $\eta$ restricts on the embedded $H^n_\CC$  to the $(n+1)$-form $\gamma$ (this implies condition \emph{2)} in Proposition \ref{Burillo}) and such that $\eta$ satisfies 
$$s_t^*\eta = t^{n+2}\eta\ .$$
The form $\gamma$ is given (with respect to the above notation for $H^n_\CC \subset H^n_\HH$) by
$$\gamma=(-1)^{n}\cdot K^*\wedge h^*_1\wedge...\wedge h^*_n$$
where for $v\in \mf h^n_{\HH}$ the symbol $v^*$ denotes the dual form of $v$.

We start the construction of $\eta$ by defining some special $n$-forms. For this let $v_{0,m}\defeq h_m^*$ and $v_{1,m}\defeq i_m^*$ for $m \in \{1,...,n\}$. Further let 
$$A^{even}\defeq \{a \in \{0,1\}^n \mid \sum_{m=1}^n a_m \text{ is even}\}$$
$$A^{odd}\defeq \{a \in \{0,1\}^n \mid \sum_{m=1}^n a_m \text{ is odd}\}$$
Then we define  the $n$-forms 
$$v(a)\defeq v_{a_1,1}\wedge v_{a_2,2} \wedge ... \wedge v_{a_n,n} \quad \text{for } a \in \{0,1\}^n$$
and assign them signs by 
$$\text{sign}(a)\defeq\begin{cases} \phantom{-} 1 \quad \text{if } \sum_{m=1}^n a_m \equiv 0 \text{ or } 3 \text{ (mod 4)}\\ -1 \quad \text{if } \sum_{m=1}^n a_m \equiv 1 \text{ or } 2 \text{ (mod 4)}\end{cases}$$
We set
$$\eta\defeq \sum_{a \in A^{even}} \text{sign}(a)v(a) \wedge K^* - \sum_{a \in A^{odd}} \text{sign}(a)v(a) \wedge J^*$$
The only form $v(a)$, which is not zero when restricted to $H^n_\CC$, is 
$$v(0,...,0)=h_1^*\wedge h_2^*\wedge ...\wedge h_n^* $$ 
As the sign of this form is ``$+$'', the form $\eta$ coincides with $\gamma$ on $H^n_\CC$.\\
Further holds
$$s_t^*\eta = t^{n+2} \eta$$
as each summand in $\eta$ consists of $n$ dual forms of vectors of the first layer of the grading of the Lie algebra $\mf h^n_\HH$ which scale linearly and one dual form of a vector of the second layer of the grading which scales quadratically.\
It remains to show that $\eta$ is closed. For this purpose we will use the bijections 
$$\tau_r: A^{odd} \to A^{even}, (a_1,...,a_r,...,a_n) \mapsto (a_1,...,1-a_r,...,a_n).$$
Then:
\begin{align*}
(-1)^n\cdot\text{d} \eta=& \sum_{a \in A^{even}} \text{sign}(a)v(a) \wedge \text d K^* - \sum_{a \in A^{odd}} \text{sign}(a)v(a) \wedge\text d J^*\\
=&\sum_{a \in A^{even}\atop r=1,...,n} \text{sign}(a)v(a) \wedge (j_r^*\wedge i_r^* + k_r^* \wedge h_r^*)\\
&\hspace*{2cm}-\sum_{a \in A^{odd}\atop r=1,...,n} \text{sign}(a)v(a) \wedge (i_r^*\wedge k_r^* + j_r^* \wedge h_r^*)\\
=& \sum_{a \in A^{even}\atop r=1,...,n} \Big( \text{sign}(a)v(a) \wedge (j_r^*\wedge i_r^* + k_r^* \wedge h_r^*)\\
&\hspace*{2cm}- \text{sign}(\tau_r(a))v(\tau_r(a)) \wedge (i_r^*\wedge k_r^* + j_r^* \wedge h_r^*)\Big)
\end{align*}
Each of term of the sum is zero as the following computations show:

We treat the case $a_r=1$, the case $a_r=0$ works analogue. As $a \in A^{even}$ we have $\sum_{m} a_m \equiv 0 \text{ or } 2 \text{ (mod 4)}$ and therefore $\sum_{m} \tau_r(a)_m \equiv 3 \text{ or } 1 \text{ (mod 4)}$ as the value at position $r$ changes from $1$ to $0$. By the definition of the sign we have $\text{sign}(a)=\text{sign}(\tau_r(a))$. As $a_r=1$ the form $v(a)$ contains $i_r^*$ (and not $h_r^*$) and the form $v(\tau_r(a))$ contains $h_r^*$ (and not $i_r^*$). Therefore such a summand looks like
$$\text{sign} (a)\cdot \big(v(a) \wedge k_r^* \wedge h_r^*- v(\tau_r(a)) \wedge i_r^*\wedge k_r^* \big) $$
The forms $v(a)$ and $v(\tau_r(a))$ only differ at position $r$. To see that $v(a) \wedge k_r^* \wedge h_r^*$ and $v(\tau_r(a)) \wedge i_r^*\wedge k_r^*$ coincide, we have to bring all index $r$ terms at coinciding positions and compare the generated signs. We do this by bringing $i_r^*$ to the last position, $h_r*$ to the second-to-last position and $k_r^*$ the third-to-last position. In  the case of $v(a) \wedge k_r^* \wedge h_r^*$ this changes the sign by $(-1)^{n-r+2} \cdot 1 \cdot 1=(-1)^{n-r+2}$. In  the case of $v(\tau_r(a)) \wedge i_r^*\wedge k_r^*$ this changes the sign by $(-1)\cdot(-1)^{n-r+1}\cdot 1=(-1)^{n-r+2}$. So the two signs are equal and therefore the two forms coincide.

This implies
$$\text d \eta =0$$
So $\eta$ is closed and the conditions of Proposition \ref{Burillo} are fulfilled. Consequently we get the lower bound $F^{n+1}_{H^n_\HH}(l) \succcurlyeq l^\frac{n+2}{n}$.
\end{proof}

Together with  the upper bound $F^{n+1}_{H^n_\HH}(l) \preccurlyeq l^\frac{n+2}{n}$ from \cite{Gruber1} this proves Theorem \ref{Thm1}.


\setlength{\bibsep}{0em}

\bibliography{bib}
\bibliographystyle{plain}

\vfill
\quad\\
\small{
Moritz Gruber\\
Fakult\"at f\"ur Mathematik\\
Karlsruhe Institute of Technology\\
Englerstraße 2\\
76131 Karlsruhe\\
Germany\\
\emph{E-mail address:} \texttt{moritz.gruber@kit.edu}

}

\end{document}